\documentclass[12pt,a4paper]{article}      

\usepackage{graphicx}
\usepackage{enumerate}
\usepackage{amssymb, amsmath, amsthm}
\usepackage[T1]{fontenc}
\usepackage{epsfig}
\usepackage{bbm}
\usepackage{color}
\usepackage{mathrsfs}
\usepackage{eurosym}
\usepackage{caption}
\usepackage{float}
\usepackage[left=2.7cm, top=2.7cm,bottom=2.7cm,right=2.7cm]{geometry}

\newtheorem{satz}{Theorem}[section]
\newtheorem{lem}[satz]{Lemma}

\newtheorem{cor}[satz]{Corollary}

\newtheorem{anm}[satz]{Remark}

\usepackage[T1]{fontenc}
\usepackage{lmodern}

\newcommand{\N}{\ensuremath{{\mathbb N}}}

\newcommand{\R}{\ensuremath{{\mathbb R}}}
\newcommand{\e}{\varepsilon}
\DeclareMathOperator{\AveP}{\underset{\pi}{Ave}}
\DeclareMathOperator{\AvePP}{\underset{\pi,\sigma}{Ave}}
\DeclareMathOperator{\AvePPP}{\underset{\pi,\sigma,\eta}{Ave}}
\DeclareMathOperator{\AvePs}{\underset{\sigma}{Ave}}
\DeclareMathOperator{\AveV}{\underset{\e}{Ave}}

\newcommand{\abs}[1]{\left\lvert#1 \right\rvert}
\newcommand{\norm}[1]{\left \lVert#1 \right\rVert}

\DeclareMathOperator{\dbm}{d_{BM}}
\DeclareMathOperator{\dint}{d}

\begin{document}

\title{Embeddings of Orlicz-Lorentz spaces into $L_1$}
\author{Joscha Prochno 
}
\date{\today}
\maketitle

\begin{abstract} 
In this article, we show that Orlicz-Lorentz spaces $\ell^n_{M,a}$, $n\in\N$ with Orlicz function $M$ and weight sequence $a$ are uniformly isomorphic to subspaces of $L_1$ if the norm $\|\cdot\|_{M,a}$ satisfies certain Hardy-type inequalities. This includes the embedding of some Lorentz spaces $\dint^n(a,p)$. Our approach is based on combinatorial averaging techniques and we prove a new result of independent interest that relates suitable averages with Orlicz-Lorentz norms.
\end{abstract}


\section{Introduction}\label{intro}

The classical Banach space $L_1$ has a very rich structure and plays an important role in many areas of modern mathematics such as Asymptotic Geometric Analysis, Functional Analysis, Harmonic Analysis, or Probability Theory. It is therefore natural to investigate and seek to understand the geometric structure of this space. This can be done, for instance, by studying its finite-dimensional subspaces. The methods used typically share a lively interplay of geometric, analytic, combinatorial, and probabilistic ideas.
While the approach via finite-dimensional subspaces is a purely `local' one, it does however bear significant information on the `global' structure of the space $L_1$. We refer the reader to the pioneering work \cite{LP1968} of Lindenstrauss and Pe\l czy\'nski in which they made Grothendieck's R\'esum\'e \cite{G1953} accessible to a broader community and launched the Local Theory of Banach spaces.

It comes as no surprise that the influential work of Lindenstrauss and Pe\l czy\'nski triggered an extensive research activity around the local structure of Banach spaces and $L_1$ in particular, and that various deep results and powerful methods have been obtained and developed. For instance, Kwapie\'n and Sch\"utt proved that all spaces whose norms are averages of $2$-concave Orlicz norms embed into $L_1$, which was a local version of a result previously obtained by Bretagnolle and Dacunha-Castelle \cite{BDC} for the infinite-dimensional setting. This characterization gives, in a certain sense, a complete picture of which spaces with a symmetric basis embed into $L_1$, but as can be seen in the case of Lorentz spaces \cite{key-Sch2} it is far from easy to apply. Therefore, it is desirable to obtain conditions which may be `easily' verified. Let us give some references most relevant to this article. It was shown by Sch\"utt in \cite[Corollary 3]{Sch1995}, using combinatorial tools partly developed in \cite{key-K-S1} (see also \cite{LPP2015}), that every Orlicz space with a $2$-concave Orlicz function embeds into $L_1$.  A similar result was later obtained by Lechner, Passenbrunner, and Prochno showing that a $2$-concave Orlicz space $\ell_M^n$ embeds into $\ell_1^{cn^3}$ with absolute constant $c\in(0,\infty)$ \cite{LPP2017} and the embedding for $2$-concave Musielak-Orlicz spaces was proved in \cite[Corollary 1.3]{key-Pro3}. In \cite{key-Sch2}, also using a combinatorial approach, the Lorentz spaces isomorphic to a subspace of $L_1$ were characterized by Sch\"utt. In the works \cite{APP2016, PS2012, S2013} embeddings of certain matrix spaces into $L_1$ were obtained.

It is natural to ask whether similar embedding results or characterizations can be shown for more general classes of Banach sequence spaces (with a symmetric basis), keeping in mind that already the cases of Orlicz and Lorentz spaces required sophisticated ideas and technical finesse. This paper is a contribution towards this goal and we shall show that certain Orlicz-Lorentz spaces, hybrids combining both Orlicz and Lorentz spaces, are uniformly isomorphic to subspaces of $L_1$. The approach we choose is based on combinatorial methods, which are of independent interest. As powerful as these averaging techniques are, as technical they can typically be, and only few experts are really familiar with them.

In order to present the main result of this article, let us briefly introduce some notation. For details, we refer the reader to Section \ref{sec:prelim} or the standard literature on Banach space geometry \cite{LT1977, LT1979, P1989, TJ1989}. We shall denote by $\dbm$ the Banach-Mazur distance, which is a measure for the geometric similarity or difference of two isomorphic spaces. We define an Orlicz-Lorentz space $\ell^n_{M,a}$ with non-increasing weight sequence $a_1\geq a_2 \geq \ldots \geq a_n \geq 0$ and Orlicz function $M$ to be $\R^n$ with the norm
  \[
      \norm{x}_{M,a} = \inf \left\{ \rho > 0 : \sum_{i=1}^n M\left(\frac{a_i x_i^*}{\rho}\right) \leq 1 \right \},
  \]
where $(x_i^*)_{i=1}^n$ is the non-increasing rearrangement of the coordinates of the vector $x$. The standard unit vector basis is a symmetric basis for this space. We shall write $\mathfrak{S}_n$ for the symmetric group of all permutations of the set $\{1,\dots,n\}$. Lastly, we denote by $L_1^{n!^32^{2n}}$ the space 
\begin{align*}
L_1^{n!^32^{2n}}
     & =  \Big\{ \big(x(\pi,\sigma,\varepsilon,\delta)\big)_{\pi,\sigma,\varepsilon,\delta}\in\R^{n!^32^{2n}} \,:\,
    \e,\delta\in\{-1,1\}^n,
    \pi,\sigma\in\mathfrak{S}_n \Big\}
\end{align*}
with the norm
  $$
    \norm{x}_1 = \frac{1}{n!^32^{2n}} \sum_{\pi,\sigma,\e,\delta} \abs{x(\pi,\sigma,\e,\delta)},\qquad x = \left(x(\pi,\sigma,\varepsilon,\delta)\right)_{\pi,\sigma,\varepsilon,\delta}\in\R^{n!^32^{2n}}.
  $$

The following theorem shows that for any $p\in(1,2)$, every Orlicz function $M$ for which $M(t)/t^{p-\epsilon}$ is decreasing ($\epsilon>0$), and all weight sequences $a$ for which the Orlicz-Lorentz norm satisfies two Hardy-type inequalities, the corresponding sequence of Orlicz-Lorentz spaces $\ell_{M,a}^n$, $n\in\N$ embeds uniformly into $L_1$.

\begin{satz}\label{thm: embedding r-convex}
  Let $n\in\N$ and $1<p<2$. Let $a_1\geq \ldots \geq a_n>0$ and for some $\epsilon\in(0,p-1)$ let $M$ be an Orlicz function such $M(t)/t^{p-\epsilon}$ is non-increasing. Assume that, for all $x\in\R^n$, the Orlicz-Lorentz norm satisfies
    \begin{equation}\label{Hardy 1}
      \bigg\| \bigg(\Big(\frac{1}{k}\sum_{i=1}^k|x_i^*|^p \Big)^{1/p}\bigg)_{k=1}^n \bigg\|_{M,a} \leq C_1 \,\|x\|_{M,a}
    \end{equation}
    and
    \begin{equation}\label{Hardy 2}
      \bigg\| \bigg(\Big(\frac{1}{k}\sum_{i=k+1}^n|x_i^*|^2 \Big)^{1/2}\bigg)_{k=1}^n \bigg\|_{M,a} \leq C_2\, \|x\|_{M,a},  
    \end{equation}
    where $C_1,C_2\in(0,\infty)$ are absolute constants. Then there is a subspace $Y_n$ of $L_1^{n!^32^{2n}}$ with $\dim{Y_n}=n$ such that $\dbm(\ell^n_{M,a},Y_n) \leq D$, where $D\in(0,\infty)$ is a constant depending only on $p$.  
\end{satz}


The parameter $p$ appears in Theorem \ref{thm: embedding r-convex}, because in our proof we need to pass through an $\ell_p$ space to embed into $L_1$ and this seems to be a technical matter. It can be seen as a kind of convexification argument. A condition similar to $M(t)/t^{p-\epsilon}$ is decreasing has already been used in \cite{PS2012}. This condidtion is satisfied, for instance, if $M$ is an $r$-concave Orlicz function with $r<p$, i.e., if $t\mapsto M(t^{1/r})$ is a concave function. Roughly speaking this assumption separates the Orlicz function $M$ and the Orlicz function $t\mapsto t^p$. We refer to the discussion in \cite{PS2012} where it is also explained that `decreasing' can be relaxed to `pseudo-decreasing'. As we shall see, the rearrangement inequalities \eqref{Hardy 1} and \eqref{Hardy 2} will appear in a natural way. They regulate the interplay of the Orlicz function and the decay of the weight sequence. Such Hardy-type inequalities have been used in a different context in \cite{J2006} and appeared in disguise already in the work \cite{key-Sch2}. As we shall see later, in the Lorentz spaces $\dint(a,q)$ those estimates are valid whenever the weight sequence has a certain regular decay, which is neither too fast nor too slow. This unfortunately also prevents us from recovering Sch\"utt's result for Orlicz spaces from \cite{Sch1995}. However, we add a rich class of new symmetric Banach sequence spaces that embed uniformly into $L_1$.
The proof of Theorem \ref{thm: embedding r-convex} is based on averaging techniques and a new combinatorial result related to Orlicz-Lorentz norms, which generalizes results from \cite{key-K-S1}, \cite{key-Sch2}, and \cite{Sch1995}. It is interesting to note that such combinatorial tools, classically used to obtain embeddings into $L_1$, recently proved to be useful in relation to the variance conjecture on hyperplane projections of $\ell_p^n$-balls \cite{AB2018} or in studying the geometry of Banach spaces between polytopes and zonotopes \cite{GLSW2002} (see also \cite{AGP2017} for a randomized version). We therefore believe that they are interesting in their own right. For further applications of these combinatorial methods in Banach space theory, we refer the reader to, e.g., \cite{key-JMST}, \cite{key-K-S2}, \cite{key-MSS}, \cite{key-Pro2}, \cite{key-Sch4} and the references cited therein. 

\section{Preliminaries}\label{sec:prelim}

We briefly present the notions and some background material used throughout this text. We split this part into several smaller subsections.

\subsubsection*{Basic notions from Banach spaces theory}

Let $X$ and $Y$ be isomorphic Banach spaces. We say that they are
$C$-isomorphic if there is an isomorphism $T:X\rightarrow Y$ with
$\|T\|\|T^{-1}\|\leq C$.
We define the Banach-Mazur distance of $X$ and $Y$ by
    $$
      \dbm(X,Y) = \inf\left\{ \|T\|\|T^{-1}\| \,:\, T\in L(X,Y) ~ \hbox{ isomorphism} \right\}.
    $$
 Let $(X_n)_n$ be a sequence of $n$-dimensional normed spaces and let $Z$ be another normed space. If there exists a constant $C\in(0,\infty)$, such that for all $n\in\N$ there exists a normed space $Y_n \subseteq Z$ with $\dim(Y_n)=n$ and $\dbm(X_n,Y_n)\leq C$, then we say that $(X_n)_n$ embeds uniformly into $Z$ or in short: $X_n$ embeds into $Z$. For a detailed introduction to the concept of Banach-Mazur distances, we refer the reader to \cite{TJ1989}.  
 
 Let $X$ be a Banach space with basis $\{e_i\}_{i=1}^n$. We call the basis $C$-symmetric 
 if there exists a constant $C\in(0,\infty)$ such that for all signs $\e_i =\pm 1$, all sequences $(a_i)_{i=1}^n$ in $\R$, and all permutations $\pi\in\mathfrak{S}_n$
   $$
     \norm{\sum_{i=1}^n a_i e_i}_X \leq C \norm{\sum_{i=1}^n \e_i a_{\pi(i)}e_i}_X.
   $$
 
We shall use the asymptotic notation $a\approx b$ to express that there exist two positive absolute constants $c_1, c_2$ such that $c_1a\leq b\leq c_2 a$ and similarly use $a\lesssim b$ or $a \gtrsim b$. If the constants depend on a parameter $\alpha$, we indicate this by $a\approx_\alpha b$, $a\lesssim_\alpha b$, or $a \gtrsim_\alpha b$. We shall use the notations $\AveP$ and $\AveV$ to denote the averages $(n!)^{-1}\sum_{\pi\in\mathfrak{S}_n}$ and $2^{-n}\sum_{\e\in\{-1,1\}^n}$, respectively. For a parameter $p$, we shall denote by $p^*$ its conjugate for which the relation $\frac{1}{p}+\frac{1}{p^*}=1$ is satisfied.

\subsubsection*{Orlicz spaces}

A convex function $M:[0,\infty) \to [0,\infty)$ with $M(0)=0$ and $M(t)>0$ for $t>0$ is called an Orlicz function. An Orlicz function (as we define it) is bijective and continuous on $[0,\infty)$. We assume that 
an Orlicz function satisfies
  \[
  \lim_{t\to0}\frac{M(t)}{t}=0
  \qquad\mbox{and}\qquad
  \lim_{t\to\infty}\frac{M(t)}{t}=\infty,
  \]
which is typically called an $N$-function. The Orlicz space $\ell^n_M$ is defined as the space $\R^n$ equipped with the Luxemburg norm
 \[
    \norm{x}_M = \inf\left\{ \rho>0 : \sum_{i=1}^n M\left(\frac{\abs{x_i}}{\rho}\right) \leq 1  \right\}.
 \]
An Orlicz function $M$ is called $r$-concave if $t\mapsto M(t^{1/r})$ is a concave function. 
Given an Orlicz function $M$, we define its conjugate function $M^*$ by the Legendre-Transform, i.e.,
 \[
    M^*(x) = \sup_{t\in[0,\infty)}\big(xt-M(t)\big).
 \]
Again, $M^*$ is an Orlicz function and $M^{**}=M$. For instance, taking $M(t)=\frac{1}{p}t^p$, $p\geq 1$, the conjugate function is given by $M^*(t)=\frac{1}{p^*}t^{p^*}$ with $\frac{1}{p^*}+\frac{1}{p}=1$. Notice also that the norm of the dual space $(\ell_M^n)^*$ is equivalent to $\norm{\cdot}_{M^*}$. 
Moreover, one has the duality relation 
\[
t\leq M^{-1}(t)(M^*)^{-1}(t) \leq 2t
\]
for all $t\geq 0$ (see, e.g., \cite{BK1994}).
We say that two Orlicz functions $M$ and $N$ are equivalent
if there are positive constants $a$ and $b$ such that for all
$t\geq0$
$$
 M(at) \leq N(t) \leq M(bt)
$$
which is equivalent to
$$
aN^{-1}(t)\leq M^{-1}(t)\leq bN^{-1}(t).
$$
If two Orlicz functions are equivalent so are their norms.
{Notice that it is enough for the functions $M$ and $N$ to be equivalent in a neighborhood of $0$ for the corresponding sequence spaces $\ell_M$ and $\ell_N$ to coincide  \cite{LT1977}. For a detailed and thorough introduction to the theory of Orlicz spaces, we refer the reader to \cite{key-Kras}, \cite{key-O}, or \cite{key-RR}.

\subsubsection*{Lorentz spaces}

Let $1\leq p <\infty$ and $a_1\geq a_2 \geq \ldots \geq a_n \geq 0$. We define the Lorentz space $\dint^n(a,p)$ to be $\R^n$ equipped with the norm
   \[
      \norm{x}_{\dint^n(a,p)} = \left( \sum_{i=1}^n a_i x_i^{*p} \right)^{1/p},
   \]
where $(x_i^*)_{i=1}^n$ is the non-increasing rearrangement of $(\abs{x_i})_{i=1}^n$. For $a_i\equiv 1$, we simply obtain the space $\ell_p^n$ and for the special choice $a_i=i^{p/q-1}$, $1\leq p <q <\infty$, as a weight sequence, we denote the corresponding Lorentz spaces by $\ell_{q,p}^n$ and write $\|\cdot \|_{q,p}$ for their norms. Using the Cauchy condensation device, one can easily show that for all $1\leq p_1 \leq p_2 \leq q<\infty$ and every $x\in\R^n$,
\[
\|x\|_{q,p_2} \leq 2^{1/p_1}\,\|x\|_{q,p_1}.
\] 
For more information on the geometry of Lorentz spaces, we refer the reader to \cite{key-L} or \cite{Pietsch1987}, and for a presentation in the context of interpolation theory we refer to \cite{BennettSharpley1988}.

\subsubsection*{Orlicz-Lorentz spaces}

Glueing together Orlicz and Lorentz spaces in an $\ell_M$-fashion, we define an Orlicz-Lorentz space $\ell^n_{M,a}$ with weight sequence $a_1\geq a_2 \geq \ldots \geq a_n \geq 0$ and Orlicz function $M$ to be $\R^n$ with the norm
  \[
      \norm{x}_{M,a} = \inf \left\{ \rho > 0 : \sum_{i=1}^n M\left(\frac{a_i x_i^*}{\rho}\right) \leq 1 \right \}.
  \]
The standard unit vector basis in $\R^n$ is a $1$-symmetric basis for these spaces and turns them into symmetric Banach spaces. For the choice $M(t)=t^p$, we obtain the Lorentz space $\dint^n(a,p)$, and for $a_i\equiv 1$ the space $\ell_{M,a}^n$ is simply the Orlicz space $\ell_M^n$. These spaces have a rich structure and we refer the reader to, for instance, the work of Montgomery-Smith \cite{MS1992}. 

\subsubsection*{Combinatorial results}
  
As we have already mentioned, our proof is based on a combinatorial approach. We will need the following deep result obtained by the author and Carsten Sch\"utt in their joined work \cite[Lemma 2.6]{PS2012}.

\begin{lem} \label{THM JoschaCarsten 1}
  Let $n\in\N$, $1\leq p < r <\infty$ and $a\in\R^n$ such that $a_1 \geq a_2 \geq \ldots \geq a_n > 0$. There exists an Orlicz function 
  $N$ such that for   all $\ell=1,\ldots,n$
    $$
      N^{*-1}\left(\frac{\ell}{n}\right) \approx \left( \frac{\ell}{n} \right)^{1/p^*}\left(\frac{1}{n}\sum_{i=1}^{\ell} |a_i|^p \right)^{1/p} 
      + \left( \frac{\ell}{n} \right)^{1/r^*}\left(\frac{1}{n}\sum_{i=\ell+1}^n |a_i|^r\right)^{1/r}.
    $$
For all such Orlicz functions and all $x\in\R^n$, we have
  $$
    c_1(r,p) \norm{x}_N \leq \left( \AveP \left( \sum_{i=1}^n |x_ia_{\pi(i)}|^r \right)^{p/r} \right)^{1/p}
    \leq c_2(r,p) \norm{x}_N,
  $$
where $c_1(r,p), c_2(r,p)\in(0,\infty)$ are constants depending only on $r$ and $p$.    
\end{lem}

For the better understanding of the combinatorial methods involved, let us show how an $\ell_p$-norm is generated by averaging over permutations. We will use this result frequently throughout this text and therefore include a proof. 

\begin{cor} \label{EXA a erzeugt p-Norm}
  Let $n\in\N$ and $1 < p <2$. The vector $(a_i)_{i=1}^n = \big((n/i)^{1/p}\,\big)_{i=1}^n$ generates the $\ell_p$-norm, i.e., for 
  all $x\in\R^n$,
    $$
      c_1(p)\norm{x}_p \leq \AveP \left( \sum_{i=1}^n |x_ia_{\pi(i)}|^2\right)^{1/2} \leq c_2(p) \norm{x}_p,
    $$
  where $c_1(p),c_2(p)\in(0,\infty)$ are constants depending only on $p$.
\end{cor}
\begin{proof}
  First, we observe that, for all $m\leq n$,
    $$
       \frac{1}{n} \sum_{i=1}^m a_i +
      \left(\frac{m}{n}\right)^{1/r^*}\left( \frac{1}{n} \sum_{i=m+1}^n \abs{a_i}^r \right)^{1/r}
      \geq \frac{1}{n} \sum_{i=1}^m a_i,
    $$
  where
    $$
      \frac{1}{n} \sum_{i=1}^m a_i   =  \frac{1}{n} \sum_{i=1}^m \left(\frac{n}{i}\right)^{1/p}
       =   \frac{1}{n^{1/p^*}} \sum_{i=1}^m i^{-1/p} = \frac{1}{n^{1/p^*}} \left( 1 + \sum_{i=2}^m i^{-1/p} \right).
    $$
  A standard integral estimate shows that
    $$
      \frac{1}{n} \sum_{i=1}^m \left(\frac{n}{i}\right)^{1/p} \leq \frac{2}{1-\frac{1}{p}} \left(\frac{m}{n}\right)^{1/p^*}.
    $$
  Furthermore, a similar argument shows that, for all  $m\geq 2$,
   \begin{align*}
      \sum_{i=1}^m i^{-1/p} 
      \geq \frac{1}{1-\frac{1}{p}} \frac{m^{1/p^*}}{2}\,.
    \end{align*}
  Hence, for all $m\leq n$
    $$
      \frac{1}{n} \sum_{i=1}^m \left(\frac{n}{i}\right)^{1/p} \geq \frac{1}{2(1-\frac{1}{p})} \left(\frac{m}{n}\right)^{1/p^*},
    $$
  since the case $m=1$ is obvious. Therefore, we have for all $m\leq n$
    $$
      \frac{1}{2(1-\frac{1}{p})} \left(\frac{m}{n}\right)^{1/p^*} \leq \frac{1}{n} \sum_{i=1}^m \left(\frac{n}{i}\right)^{1/p}
      \leq \frac{2}{1-\frac{1}{p}} \left(\frac{m}{n}\right)^{1/p^*}
    $$
  and hence
    $$
       \frac{1}{n} \sum_{i=1}^m a_i +
      \left(\frac{m}{n}\right)^{1/r^*}\left( \frac{1}{n} \sum_{i=m+1}^n \abs{a_i}^r \right)^{1/r}
      \geq \frac{1}{2(1-\frac{1}{p})} \left(\frac{m}{n}\right)^{1/p^*}.
    $$
  In addition, we have
    $$
      \left(\frac{m}{n}\right)^{1/2}\left(\frac{1}{n}\sum_{i=m+1}^n \abs{a_i}^2\right)^{1/2}  =
      \frac{\sqrt{m}}{n^{1/p^*}} \left( \sum_{i=m+1}^n i^{-2/p}\right)^{1/2}.
    $$
  Again, a simple integral estimate yields
    $$
      \sum_{i=m+1}^n i^{-2/p} 
      \leq \frac{1}{\frac{2}{p}-1} \left(m^{-2/p+1} - n^{-2/p+1}\right),
    $$
  and so we have
    $$
      \left( \sum_{i=m+1}^n i^{-2/p}\right)^{1/2} \leq \left(\frac{1}{\frac{2}{p}-1}m^{-2/p+1} \right)^{1/2}
      = \left(\frac{1}{\frac{2}{p}-1}\right)^{1/2}m^{-1/p+1/2}.
    $$
  Therefore,
    $$
      \frac{\sqrt{m}}{n^{1/p^*}} \left( \sum_{i=m+1}^n i^{-2/p}\right)^{1/2} \leq
      \frac{\sqrt{m}}{n^{1/p^*}} \left(\frac{1}{\frac{2}{p}-1}\right)^{1/2}m^{-1/p+1/2}
      = \left(\frac{1}{\frac{2}{p}-1}\right)^{1/2}\left(\frac{m}{n}\right)^{1/p^*}.
    $$
  Hence, for all $m\leq n$
    \begin{eqnarray*}
       \frac{1}{n} \sum_{i=1}^m a_i +
      \left(\frac{m}{n}\right)^{1/2}\left( \frac{1}{n} \sum_{i=k+1}^n \abs{a_i}^2 \right)^{1/2}
      & \leq & \frac{2}{1-\frac{1}{p}} \left(\frac{m}{n}\right)^{1/p^*}
      + \left(\frac{1}{\frac{2}{p}-1}\right)^{1/2}\left(\frac{m}{n}\right)^{1/p^*} \\
      & = & \left(\frac{2}{1-\frac{1}{p}} + \frac{1}{\sqrt{\frac{2}{p}-1}}\right)\left(\frac{m}{n}\right)^{1/p^*},
    \end{eqnarray*}
    which concludes the proof.
\end{proof}

Another combinatorial tool that we shall use is the following result of Sch\"utt taken from~\cite[Lemma 2.3]{key-Sch2}. 

\begin{lem}\label{KOR Groessenordnung Mittel und 2Norm pNorm mix}
  Let $n\in\N$ and $1\leq p \leq 2$. Then, for all $x\in\R^{n}$ and any $k=1.\dots,n$, 
    \begin{align*}
    \Bigg( \AveP \bigg(\sum_{i\leq \frac{n}{k}}\abs{x_{\pi(i)}}^2\bigg)^{p/2}\Bigg)^{1/p}
     & \approx \bigg(\frac{1}{k}\sum_{i=1}^k|x_i^*|^p\bigg)^{1/p}
    +  \bigg(\frac{1}{k} \sum_{i=k+1}^{n}|x_i^*|^2\bigg)^{1/2},
  \end{align*}
 where $x_i^*$, $i=1,\dots,n$ denotes the non-increasing rearrangement of the numbers $|x_i|$, $i=1,\dots,n$. 
\end{lem}

Readers familiar with interpolation theory might recognize expressions of this type. Indeed, a result of this flavor has also been obtained via real interpolation techniques by Lechner, Passenbrunner, and Prochno in \cite{LPP2015}.

\section{Combinatorial elements of the proof}

Before we come to the embedding of Orlicz-Lorentz spaces into $L_1$, we require a new combinatorial averaging device. We prove that an abstract Orlicz-Lorentz norm can be generated by permutation-averages of an $\ell_2$-norm of suitably chosen vectors. Recall that to any vector $d\in\R^n$ whose coordinates are non-increasing there corresponds an Orlicz function as shown in Lemma \ref{THM JoschaCarsten 1}.

\begin{satz}\label{THM combinatorial result}
  Let $1<p<2$, $z=\big(( n/i)^{1/p}\big)_{i=1}^n$, and $d_1\geq \ldots \geq d_n>0$. Let $a_1\geq \ldots \geq a_n>0$ and consider for $k=1,\dots,n$ the vectors $c^k=(a_k,\dots,a_k,0,\dots,0)\in\R^n$ with $n/k$ non-zero entries. Assume that, for all $x\in\R^n$,
      \begin{equation}\label{Hardy 1}
        \bigg\| \bigg(\Big(\frac{1}{k}\sum_{i=1}^k|x_i^*|^p \Big)^{1/p}\bigg)_{k=1}^n \bigg\|_{M_d,a} \leq C_1 \,\|x\|_{M_d,a}
      \end{equation}
      and
      \begin{equation}\label{Hardy 2}
        \bigg\| \bigg(\Big(\frac{1}{k}\sum_{i=k+1}^n|x_i^*|^2 \Big)^{1/2}\bigg)_{k=1}^n \bigg\|_{M_d,a} \leq C_2\, \|x\|_{M_d,a},  
      \end{equation}
      where $M_d$ is the Orlicz function related to $d$ and $C_1,C_2\in(0,\infty)$ are absolute constants.
  Then, for all $x\in\R^n$,  
  $$
   c_1(p)\norm{x}_{M_d,a} \leq \AvePPP \left( \sum_{i,k=1}^n |x_i c^k_{\pi(i)} d_{\sigma(k)} z_{\eta(k)} |^2 \right)^{1/2} \leq c_2(p) \norm{x}_{M_d,a},
  $$
  where $c_1(p),c_2(p)\in(0,\infty)$ are constants depending only on $p$.
\end{satz}
\begin{proof}
  Let us start with the lower bound. As we know from Corollary \ref{EXA a erzeugt p-Norm}, the vector $z$ generates the $\ell_p$-norm and so
    $$
      \AvePPP \left( \sum_{i,k=1}^n |x_i c^k_{\pi(i)} d_{\sigma(k)} z_{\eta(k)} |^2 \right)^{1/2} \approx_p \AvePP \left( \sum_{k=1}^n \left( \sum_{i=1}^n | x_i c^k_{\pi(i)} d_{\sigma(k)} |^2 \right)^{p/2} \right)^{1/p}\,.
    $$
  The triangle inequality yields the estimate
    \begin{align*}
     \AvePP \left( \sum_{k=1}^n \left( \sum_{i=1}^n |x_i c^k_{i\pi(i)} d_{\sigma(k)}|^2 \right)^{p/2} \right)^{1/p}
     & \geq    \AvePs \left( \sum_{k=1}^n \bigg| \AveP \left( \sum_{i=1}^n | x_i c^k_{\pi(i)} d_{\sigma(k)} |^2 \right)^{1/2} \bigg|^p \right)^{1/p} \\
     & =  \AvePs \left( \sum_{k=1}^n |d_{\sigma(k)}|^p \bigg| \AveP \left( \sum_{i=1}^n |x_i c^k_{\pi(i)}|^2 \right)^{1/2} \bigg|^p \right)^{1/p}.
    \end{align*}
  An application of Lemma \ref{KOR Groessenordnung Mittel und 2Norm pNorm mix}, where we simply omit the quadratic term, shows that
    $$
     \AvePs \left( \sum_{k=1}^n |d_{\sigma(k)}|^p \bigg| \AveP \left( \sum_{i=1}^n |x_i c^k_{\pi(i)}|^2 \right)^{1/2} \bigg|^p \right)^{1/p}
     \approx \AvePs \left( \sum_{k=1}^n |d_{\sigma(k)}|^p \bigg| \frac{1}{k} \sum_{i=1}^k x_i^* \bigg|^p a_k^p \right)^{1/p}.
    $$
  By Lemma \ref{THM JoschaCarsten 1} there exists an Orlicz function $M_d$ such that
    $$
       \AvePs \left( \sum_{k=1}^n |d_{\sigma(k)}|^p \bigg| \frac{1}{k} \sum_{i=1}^k x_i^* \bigg|^p a_k^p \right)^{1/p} 
       \approx_p \norm{\left(a_k  \frac{1}{k} \sum_{i=1}^k x_i^* \right)_{k=1}^n}_{M_d}
       \geq  \norm{\left(a_k x_k^* \right)_{k=1}^n}_{M_d},
    $$
  where we used that $\sum_{i=1}^k x_i^* \geq k x^*_k$ for any $k=1,\ldots,n$. Therefore, we obtain the desired lower bound
    $$
       \AvePPP \left( \sum_{i,k=1}^n |x_i c^k_{\pi(i)} d_{\sigma(k)} z_{\eta(k)} |^2 \right)^{1/2} \geq c_1(p) \norm{x}_{M_d,a},
    $$
  where the constant $c_1(p)\in(0,\infty)$ depends only on $p$.    
  
  We now proceed with the upper bound. As we have seen, since $z$ generates the $\ell_p$-norm, we have
    $$
      \AvePPP \left( \sum_{i,k=1}^n |x_i c^k_{\pi(i)} d_{\sigma(k)} z_{\eta(k)} |^2 \right)^{1/2} \approx_p \AvePP \left( \sum_{k=1}^n \left( \sum_{i=1}^n | x_i c^k_{\pi(i)} d_{\sigma(k)}|^2 \right)^{p/2} \right)^{1/p}.
    $$
  It follows from Jensen's inequality that
    $$
     \AvePP \left( \sum_{k=1}^n \left( \sum_{i=1}^n |x_i c^k_{\pi(i)} d_{\sigma(k)} |^2 \right)^{p/2} \right)^{1/p} 
     \leq \AvePs \left( \sum_{k=1}^n |d_{\sigma(k)}|^p \AveP \left( \sum_{i=1}^n | x_i c^k_{\pi(i)}|^2 \right)^{p/2} \right)^{1/p}.
    $$
   Using again Lemma \ref{THM JoschaCarsten 1}, we obtain that   
    $$
      \AvePs \left( \sum_{k=1}^n |d_{\sigma(k)}|^p \AveP \left( \sum_{i=1}^n | x_i c^k_{\pi(i)}|^2 \right)^{p/2} \right)^{1/p} 
      \approx_p \norm{\left( \left( \AveP \left( \sum_{i=1}^n |x_i c^k_{\pi(i)}|^2 \right)^{p/2} \right)^{1/p} \right)_{k=1}^n}_{M_d}.
    $$
   An application of Lemma \ref{KOR Groessenordnung Mittel und 2Norm pNorm mix} and the triangle inequality yields 
    \begin{align*}
      & \norm{\left( \left( \AveP \left( \sum_{i=1}^n |x_i c^k_{\pi(i)}|^2 \right)^{p/2} \right)^{1/p} \right)_{k=1}^n}_{M_d} \cr
      &\lesssim \Bigg\|\left( a_k \left( \frac{1}{k} \sum_{i=1}^k x_i^{*p} \right)^{1/p} \right)_{k=1}^n\Bigg\|_{M_d} 
       + \Bigg\|\left( a_k \left( \frac{1}{k} \sum_{i=k+1}^n x_i^{*2} \right)^{1/2} \right)_{k=1}^n\Bigg\|_{M_d}\,.
    \end{align*}
   Therefore, conditions \eqref{Hardy 1} and \eqref{Hardy 2} show that 
   \[
   \AvePPP \left( \sum_{i,k=1}^n |x_i c^k_{\pi(i)} d_{\sigma(k)} z_{\eta(k)} |^2 \right)^{1/2}  \leq c_2(p) \norm{x}_{M_d,a},
   \]
   with constant $c_2(p)\in(0,\infty)$ depending only on $p$. This completes the proof.
\end{proof}

\begin{anm}\label{rem:hady type inequality 1}
Let us note that the first Hardy-type inequality in Theorem \ref{THM combinatorial result} (and Theorem \ref{thm: embedding r-convex}) holds, for instance, if the weight sequence $a$ defining the Orlicz-Lorentz space does not decay too slowly, i.e., if for all $k=1,\dots,n$,
\begin{equation}\label{eq:weight sequence not to slowly decaying}
  \sum_{i=k+1}^n \frac{a_i}{i^{1/p}} \leq C a_k k^{1-1/p},
\end{equation}
where $C\in(0,\infty)$ is an absolute constant. Indeed, using the Hahn-Banach theorem and that the norm of the dual space of $\ell_{M_d}^n$ is up to a factor $2$ equivalent to $\|\cdot\|_{M_d^*}$, we obtain 
    \begin{equation*}
    \norm{\left( a_k \left( \frac{1}{k} \sum_{i=1}^k x_i^{*p} \right)^{1/p} \right)_{k=1}^n}_{M_d} \approx \sup_{y\in B^n_{M^*_d}} \sum_{k=1}^n y_k \frac{a_k}{k^{1/p}} \left(\sum_{i=1}^k x_i^{*p} \right)^{1/p},
    \end{equation*}
    where $B_{M_d^*}^n$ is the unit ball of the space $\ell_{M_d^*}^n$. Since we may assume without loss of generality that $y_1 \geq \ldots \geq y_n$ and because of the Lorentz-norm estimate $\norm{\cdot}_p \leq 2 \norm{\cdot}_{p,1}$, we obtain
    \begin{align*}
      \sup_{y\in B_{M^*}} \sum_{k=1}^n y_k \frac{a_k}{k^{1/p}} \left(\sum_{i=1}^k x_i^{*p} \right)^{1/p} 
      & \leq  2 \sup_{y\in B^n_{M^*_d}} \sum_{k=1}^n y_k \frac{a_k}{k^{1/p}} \sum_{i=1}^k \frac{1}{i^{1-1/p}}x_i^{*} \cr
      & =   2 \sup_{y\in B^n_{M^*_d}} \sum_{i=1}^n \frac{1}{i^{1-1/p}}x_i^{*} \sum_{k=i}^n y_k  \frac{a_k}{k^{1/p}}  \cr
      & \leq  2 \sup_{y\in B^n_{M^*_d}} \sum_{i=1}^n \frac{1}{i^{1-1/p}}x_i^{*} y_i \sum_{k=i}^n \frac{a_k}{k^{1/p}}. 
    \end{align*}
    Applying condition \eqref{eq:weight sequence not to slowly decaying} and again the duality relation used before, we find that
     $$
       \sup_{y\in B^n_{M^*_d}} \sum_{i=1}^n \frac{1}{i^{1-1/p}}x_i^{*} y_i \sum_{k=i}^n \frac{a_k}{k^{1/p}}
       \leq C \sup_{y\in B^n_{M^*_d}} \sum_{i=1}^n \frac{1}{i^{1-1/p}}x_i^{*} y_i i^{1-1/p}a_i
       \approx \norm{\left(x_i^*a_i\right)_{i=1}^n}_{M_d}.
     $$
\end{anm}
    
\section{The embedding of Orlicz-Lorentz spaces into $L_1$}

We are now prepared to present the proof of Theorem \ref{thm: embedding r-convex}, which is based on Theorem \ref{THM combinatorial result}. In fact, we need to prove two things. One is that we can choose the vector $d$ such that it yields an Orlicz function equivalent to the one from Theorem \ref{THM combinatorial result}. The second is that the average over permutations from Theorem \ref{THM combinatorial result} is equivalent to an $L_1$-norm. The latter will be a consequence of Khintchine's inequality.

\begin{proof}[Proof of Theorem \ref{thm: embedding r-convex}]
  We consider $z=\big((n/i)^{1/p}\big)_{i=1}^n$ and choose the sequence $d$ such that, for all $\ell=1,\dots,n$,
  \[
  (M^*)^{-1}\Big(\frac{\ell}{n}\Big) = \frac{1}{n} \sum_{i=1}^\ell d_i.
  \]
  Let $c^1,\dots,c^n\in\R^n$ be the vectors as chosen in Theorem \ref{THM combinatorial result}. Then Theorem \ref{THM combinatorial result} shows that 
      \[
        \AvePPP \left( \sum_{i,k=1}^n |x_i c^k_{\pi(i)} d_{\sigma(k)}z_{\eta(k)}|^2 \right)^{1/2}
        \approx_p \norm{x}_{M_d,a},
      \]
    where $M_d$ is an Orlicz function such that for all $\ell \leq n$
      \[
        (M_d^*)^{-1}\left( \frac{\ell}{n} \right) \approx \frac{1}{n} \sum_{i=1}^{\ell} |d_i|
        +\left( \frac{\ell}{n} \right)^{1/p^*}\left(\frac{1}{n} \sum_{i=\ell+1}^n |d_i|^{p} \right)^{1/p}.
      \]
   We show that $M_d^*$ and $M^*$ are equivalent Orlicz functions. The lower bound is immediate, because of the choice of the sequence $d$ and so we obtain
   \[
   (M_d^*)^{-1}\left( \frac{\ell}{n} \right) \gtrsim \frac{1}{n} \sum_{i=1}^{\ell} |d_i| = (M^*)^{-1}\left( \frac{\ell}{n} \right).
   \]  
   Let us proceed with the upper bound. We have
   \[
    (M_d^*)^{-1}\left( \frac{\ell}{n} \right) \lesssim (M^*)^{-1}\left( \frac{\ell}{n} \right)
     +\left( \frac{\ell}{n} \right)^{1/p^*}\left(\frac{1}{n} \sum_{i=\ell+1}^n |d_i|^{p} \right)^{1/p}
   \]
   and so it remains to estimate the second term on the right-hand side. First, we observe that for all $\ell=1,\dots,n$
   \[
   (M^*)^{-1}\left( \frac{\ell}{n} \right) \geq \frac{\ell}{n}\,d_\ell.
   \]
  Using the relation $t\leq M^{-1}(t)(M^*)^{-1}(t) \leq 2t$ that holds for all $t\geq 0$, we obtain that, for all $\ell=1,\dots,n$,
  \[
  d_\ell \leq \frac{(M^*)^{-1}(\ell/n)}{\ell/n} \leq \frac{2}{M^{-1}(\ell/n)}\,.
  \]
  Hence, for all $\ell=1,\dots,n$, 
  \[
  \frac{1}{n} \sum_{i=\ell+1}^n |d_i|^{p} \leq \frac{2^{p}}{n}\sum_{i=\ell+1}^n \Big|\frac{1}{M^{-1}(i/n)}\Big|^p
  \]
  Consider $p_\epsilon:=p-\epsilon$ for $\epsilon\in(0,p-1)$ such that $M(t)/t^{p-\epsilon}$ is decreasing. Then also $t\mapsto \frac{t}{(M^{-1}(t))^{p_\epsilon}}$ is a decreasing function. In particular, the function
  \[
  t\mapsto \Big|\frac{t}{(M^{-1}(t))^{p_\epsilon}}\Big|^{p/p_\epsilon} = \frac{t^{p/p_\epsilon}}{(M^{-1}(t))^p}
  \]
  is decreasing. Therefore, we obtain the estimate
  \begin{align*}
  \frac{1}{n} \sum_{i=\ell+1}^n |d_i|^{p} & \leq \frac{2^{p}}{n}\sum_{i=\ell+1}^n \frac{(i/n)^{p/p_\epsilon}}{(M^{-1}(i/n))^p}\,\left(\frac{n}{i}\right)^{p/p_\epsilon} 
  \leq \frac{2^{p}}{n} \frac{(\ell/n)^{p/p_\epsilon}}{(M^{-1}(\ell/n))^p}\sum_{i=\ell+1}^n \left(\frac{n}{i}\right)^{p/p_\epsilon} \cr
  & = \frac{2^{p}}{n} \frac{\ell^{p/p_\epsilon}}{(M^{-1}(\ell/n))^p}\sum_{i=\ell+1}^n i^{-p/p_\epsilon}
  \lesssim_{p}  \frac{\ell^{p/p_\epsilon}}{n\,(M^{-1}(\ell/n))^p}\,\ell^{1-p/p_\epsilon} \cr
  & = \frac{\ell/n}{(M^{-1}(\ell/n))^p}\,,
  \end{align*}
  where in the penultimate step we used that $p_\epsilon<p$. Putting everything together, we arrive at 
  \[
  \left( \frac{\ell}{n} \right)^{1/p^*}\left(\frac{1}{n} \sum_{i=\ell+1}^n |d_i|^{p} \right)^{1/p} \lesssim_p \frac{(\ell/n)^{1/p^*}(\ell/n)^{1/p}}{M^{-1}(\ell/n)} = \frac{\ell/n}{M^{-1}(\ell/n)} \leq (M^*)^{-1}\left( \frac{\ell}{n} \right),
  \]
  where in the last step we used the duality relation $t/M^{-1}(t)\leq (M^*)^{-1}(t)$ valid for all $t\geq 0$. This shows that the Orlicz functions $M_d$ and $M$ are equivalent and  we have
    \[
      \AvePPP \left( \sum_{i,k=1}^n |x_i c^k_{\pi(i)} d_{\sigma(k)}z_{\eta(k)}|^2 \right)^{1/2}
      \approx_p \norm{x}_{M,a}.
    \]  
  We now define our embedding into $L_1$ as follows,
    \[
      \Psi_n : \ell^n_{M,a} \to L_1^{n!^32^{2n}},\qquad x \mapsto \left( \sum_{i,k=1}^n x_i c^k_{\pi(i)} d_{\sigma(k)}z_{\eta(k)} \e_i
      \delta_k \right)_{\pi,\sigma,\eta,\e,\delta}.
    \]
  It is a direct consequence of Khintchine's inequality (see, e.g., \cite[Theorem 6.1]{TJ1989}) that
    \[
      \norm{\Psi_n(x)}_1 \approx \norm{x}_{M,a}.
    \] 
   This means that there exists a constant $D\in(0,\infty)$ depending only on $p$ such that for any $n\in\N$ there exists a subspace $Y_n$ of $L_1^{n!^3 2^{2n}}$ with $\dim(Y_n)=n$ such that $\dbm(\ell^n_{M,a},Y_n) \leq D$.    
\end{proof}

Let us close this article with a comment on the special case of embeddings of Lorentz spaces into $L_1$.

\begin{anm}
To obtain from our result the embeddings of certain Lorentz spaces $\dint^n(a,r)$ into $L_1$, one needs to assure that the Hardy-type inequalities are satisfied. In Remark \ref{rem:hady type inequality 1}, we have already demonstrated how the first inequality \eqref{Hardy 1} can be derived for suitably decaying weights $a$. In fact, in the case of Lorentz spaces the second inequality \eqref{Hardy 2} follows from a complementary condition. Both assumptions together assure that the decay of $a$ is regular enough. Let us consider the Orlicz function $M(t)=t^r$ with $1<r<p$ and assume that for all $k=1,\dots,n$ the sequence $a$ of weights satisfies
\begin{equation}\label{eq:weight sequence not too fastly decaying}
\sum_{i=1}^k \frac{a_i^r}{i^{p/2}} \leq C a_k^r k^{1-p/2}\,
\end{equation}
where $C\in(0,\infty)$ is an absolute constant.  Then, because of the inequality $\|\cdot\|_{2} \leq 2^{1/r} \|\cdot\|_{2,r}$, we obtain from interchanging the order of summation that
\begin{align*}
\bigg(\sum_{k=1}^n a_k^r\bigg( \frac{1}{k}\sum_{i=k+1}^nx_i^{*2}\bigg)^{r/2}\, \bigg)^{1/r} 
& \leq  \bigg(\sum_{k=1}^n \frac{a_k^r}{k^{r/2}} \|(x_i^*)_{i=k+1}^n\|_{2,r}^r\bigg)^{1/r} \cr
& = \bigg(\sum_{k=1}^n \frac{a_k^r}{k^{r/2}} \sum_{i=k+1}^n i^{r/2-1}x_i^{*r}\bigg)^{1/r} \cr
& = \bigg(\sum_{i=1}^n i^{r/2-1}x_i^{*r} \sum_{k=1}^i \frac{a_i^r}{i^{r/2}} \bigg)^{1/r} \cr
& \leq C^{1/r} \|x\|_{\dint^n(a,r)},
\end{align*}
where we used condition \eqref{eq:weight sequence not too fastly decaying} in the last step of the computation.
\end{anm}

It would be nice to obtain complete characterizations describing exactly which Orlicz-Lorentz spaces embed uniformly into $L_1$. As already the result in this paper and previous ones like \cite{key-K-S1}, \cite{PS2012}, \cite{Sch1995}, or \cite{key-Sch2} show, this question is one of considerable difficulty. One main problem is that in general Orlicz functions are not homogeneous for some parameter $\alpha$, i.e., $M(\lambda t)\neq \lambda^\alpha M(t)$. 

\subsection*{Acknowledgement}

JP has been supported by a \textit{Visiting International Professor Fellowship} from the Ruhr University Bochum and its Research School PLUS as well as by the Austrian Science Fund (FWF) Project F5508-N26, which is part of the Special Research Program ``Quasi-Monte Carlo Methods: Theory and Applications''.

\bibliographystyle{plain}
\bibliography{orlicz_lorentz}

\noindent{\bf Joscha Prochno}\\
Institute of Mathematics \& Scientific Computing\\
University of Graz\\
Heinrichstra\ss e 36\\
8010 Graz, Austria\\
e-mail: {\em joscha.prochno@uni-graz.at}\\

\end{document}